\documentclass[12pt]{iopart}
\usepackage{iopams,harvard,graphicx,ifthen}

\newenvironment{proof}{\par\smallskip\noindent{\sl Proof.\/}}{\finprf}

\newtheorem{lemma}{Lemma}

\newtheorem{theorem}[lemma]{Theorem}

\newcommand{\be}[1]{\begin{equation}\ifthenelse{\equal{#1}{}}{}{\label{#1}}}
\newcommand{\ee}{\end{equation}}
\newcommand{\N}{\mathbb{N}}
\newcommand{\R}{\mathbb{R}}
\newcommand{\C}{\mathbb{C}}

\newcommand{\nrmphi}{\|\phi\|_{L^2(\R^3)}^2}
\newcommand{\finprf}{\unskip\null\hfill$\;\square$\vskip 0.3cm}
\renewcommand{\(}{\left(}
\renewcommand{\)}{\right)}
\newcommand{\binom}[2]{{#1\choose #2}}
\newcommand{\eqn}[1]{(\ref{#1})}

\newcommand{\A}{\mathrm A}
\newcommand{\B}{\mathrm B}

\begin{document}

\title[Critical magnetic field in relativistic atomic physics]{Characterization of the critical magnetic field in the Dirac-Coulomb equation}

\author{J Dolbeault$^1$, M Esteban$^2$ and M Loss$^3$}

\address{{}$^{1,2}$ Ceremade (UMR CNRS 7534), Universit\'e Paris Dauphine, Place de Lattre de Tassigny, F-75775 Paris C\'edex 16, France\\
{}$^3$ School of Mathematics, Georgia Institute of Technology Atlanta, GA 30332, USA}
\ead{{}$^1$dolbeaul@ceremade.dauphine.fr, {}$^2$esteban@ceremade.dauphine.fr, {}$^3$loss@math.gatech.edu}

\begin{abstract}
We consider a relativistic hydrogenic atom in a strong magnetic field. The ground state level depends on the strength of the magnetic field and reaches the lower end of the spectral gap of the Dirac-Coulomb operator for a certain critical value, the {\sl critical magnetic field.\/} We also define a critical magnetic field in a Landau level ansatz.

In both cases, when the charge $Z$ of the nucleus is not too small, these critical magnetic fields are huge when measured in Tesla, but not so big when the equation is written in dimensionless form. When computed in the Landau level ansatz, orders of magnitude of the critical field are correct, as well as the dependence in $Z$. The computed value is however significantly too big for a large $Z$, and the wave function is not well approximated. Hence, accurate numerical computations involving the Dirac equation cannot systematically rely on the Landau level ansatz.

Our approach is based on a scaling property. The critical magnetic field is characterized in terms of an equivalent eigenvalue problem. This is our main analytical result, and also the starting point of our numerical scheme.
\end{abstract}


\vspace{2pc}\noindent{\it Keywords}: Relativistic quantum mechanics, ground state, min-max levels, magnetic field, Dirac equation, Dirac-Coulomb Hamiltonian, relativistic hydrogen atom, pair creation, Landau levels


\maketitle

\section{Introduction}\label{Sec:intro}

It is widely accepted in physics that eigenstates in strong magnetic fields can be well approximated in a Landau level ansatz. In this paper, we investigate the validity of such an ansatz in the case of the magnetic Dirac-Coulomb equation for large magnetic fields. More specifically, we study the {\sl critical threshold\/} for the magnetic field defined as the smallest value of the magnetic field for which the lowest eigenvalue (the ground state) in the gap of the Dirac operator reaches its lower end. After a reformulation in terms of an equivalent minimization problem, we compare the values of the critical magnetic field with and without ansatz, both from an analytical and a numerical point of view.

Our main theorem is stated in Section~\ref{Sec:mainthm}. It characterizes the critical magnetic field as a function of the lowest energy for an explicit eigenvalue problem. Section~\ref{Sec:Proofs} will be devoted to its proof. A Landau level ansatz is then defined in Section~\ref{Sec:LandauLevels} and some comparison results for the critical magnetic field, with and without ansatz, are given there. A numerical method based on Theorem~\ref{Thm:bcritformula} has been implemented in both cases. Results are given in Section~\ref{Sec:Numerics} and compared with earlier works, see \cite{SchlueterEtAl}. They show that for large coupling constants (in the electrostatic field), the critical threshold is well below the critical threshold in the Landau ansatz.

Critical magnetic fields are huge and can eventually be encountered only in some extreme situations like {\sl magnetars,\/} which are neutron stars with intense magnetic fields, see \cite{AJ99}. This is the only known domain of physics for which our computations might eventually be relevant, see \cite{DEL} for a discussion. In the Dirac-Coulomb model, the value of the critical magnetic field only provides an order of magnitude of the field strength for which pair creation could eventually occur. Such a phenomenon should of course be studied in a full QED framework and the computations done using the Dirac-Coulomb operator are only an indication on the scales that should be taken into account. See \cite{pickl} for more details. For a review on large magnetic fields in physics, see \cite{duncan}.

\section{Main results}\label{Sec:mainthm}

The magnetic Dirac operator with Coulomb potential $\nu/|x|$ can be written as :
\be{operator}
H_B:= \left( \begin{array}{cc} {\mathbb I}-\nu/|x| & -\,i\,\sigma \cdot (\nabla-i\,\A) \\ -\,i\,\sigma \cdot (\nabla-i\,\A) & -{\mathbb I}-\nu/|x| \end{array} \right)
\ee
where $\A$ is a magnetic potential corresponding to $\B$, and $\mathbb I$ and $\sigma_k$ are respectively the identity and the Pauli matrices
\be{}
\mathbb I=\left( \begin{array}{cc} 1 & 0 \\ 0 & 1 \\ \end{array} \right)\!,\;
\sigma_1=\left( \begin{array}{cc} 0 & 1 \\ 1 & 0 \\ \end{array} \right)\!,\; \sigma_2=\left( \begin{array}{cc} 0 & -i \\ i & 0 \\ \end{array}\right)\!,\; \sigma_3=\left( \begin{array}{cc} 1 & 0\\ 0 &-1\\ \end{array}\right)\! .
\ee
Let $\B=(0,0,B)$ be a constant magnetic field and $\A_B$ the associated magnetic potential. For any $x=(x_1,x_2,x_3)\in\R^3$, define
\be{}
P_B:=-\,i\,\sigma \cdot (\nabla-i\,\A_B(x))\;,\quad \A_B(x):=\frac B2\left(\begin{array}{c} -x_2\\ x_1\\ 0\end{array}\right)\,,
\ee
and consider the functional
\be{}
J[\phi,\lambda,\nu,B]:= \int_{\R^3}\left(\frac{|P_B \phi|^2} {1+\lambda+\frac\nu{|x|}}+(1-\lambda)\,|\phi|^2-\frac\nu{|x|}\,|\phi|^2\right)\,d^3x\label{jay}
\ee
on the set of admissible functions ${\mathcal A}(\nu,B):=\{\phi\in C^\infty_0(\R^3)\,:\,\nrmphi=1\,,\;\lambda\mapsto J[\phi,\lambda,\nu,B]\;\mbox{\sl changes sign in}\;(-1,+\infty)\}$. The essential spectrum of $H_B$ is $\R\setminus (-1,1)$. By \cite[Theorem 1]{DEL}, the smallest eigenvalue in $(-1,1)$ of $H_B$ is
\be{}
\lambda_1(\nu,B):=\inf_{\phi\in{\mathcal A}(\nu,B)}\lambda[\phi,\nu,B]
\ee
where $\lambda=\lambda[\phi,\nu,B]$ is either the unique solution to $J[\phi,\lambda,\nu,B]=0$ if $\phi\in{\mathcal A}(\nu,B)$, or $\lambda[\phi,\nu,B]=-1$ if $J[\phi,-1,\nu,B]\leq 0$. Also see \cite{DEL} for the relation of the two-components spinors in ${\mathcal A}(\nu,B)$ with the four-components spinors and the action of $H_B$ on them. 

The {\sl critical magnetic field\/} is defined by
\be{}
B(\nu):=\inf\left\{B>0\;:\;\liminf_{b\nearrow B}\lambda_1(\nu,b)=-1\right\}\;.
\ee
Define the auxiliary functional
\be{functionalbcrit}
\mathcal{E}_{B,\nu}[\phi]:=\int_{\R^3}\frac{|x|}{\nu}\,|P_B\,\phi|^2\,d^3x-\int_{\R^3}\frac{\nu}{|x|}\,|\phi|^2\,d^3x\;,
\ee
that is $\,\mathcal{E}_{B,\nu}[\phi]+2\,\nrmphi=J[\phi,-1,\nu,B]\,$. The scaling $\phi_B:=B^{3/4}\,\phi\left(B^{1/2}\,x\right)$ preserves the $L^2$ norm, and yields
\be{Scaling}
\mathcal{E}_{B,\nu}[\phi_B] = \sqrt B\,\mathcal{E}_{1,\nu}[\phi]\;.
\ee
We define
\be{Eqn:Equiv}
\mu(\nu):=\inf_{0\not\equiv\phi\in C^\infty_0(\R^3)}\frac{\mathcal{E}_{1,\nu}[\phi]}{\nrmphi}\;.
\ee
Recall that $\lambda_1(\nu,B)$ is characterized as an eigenvalue of $H_B$ only as long as it takes values in $(-1,1)$. If we could take the limit $B\to B(\nu)$, we would formally get that $-1=\lambda_1(\nu,B(\nu))$, which, still formally, amounts to $\inf_{0\not\equiv\phi\in C^\infty_0(\R^3)}J[\phi,-1,\nu,B]=0$. It is therefore natural to expect that $\sqrt{B(\nu)}\,\mu(\nu)+2=0$. Proving this is the purpose of our main result. 
\begin{theorem}\label{Thm:bcritformula} For all $\nu\in(0,1)$, $\mu(\nu)$ is negative, finite, 
\be{}
B(\nu)=\frac 4{\mu(\nu)^2}
\ee
and $B(\nu)$ is a continuous, monotone decreasing function of $\nu$ on $(0,1)$. \end{theorem}

\section{Proofs}\label{Sec:Proofs}

\subsection{Preliminary results}

\begin{lemma}\label{Lem6} On the interval $(0,1)$, the function $\nu\mapsto\mu(\nu)$ is continuous, monotone decreasing and takes only negative real values.\end{lemma} 
\begin{proof} The monotonicity of the function $\mu(\nu)$ is a consequence of its definition. On the other hand, the functional
\be{}
(\nu,\phi)\mapsto\nu\,\mathcal{E}_{1,\nu}[\phi]=\int_{\R^3}|x|\,|P_1\,\phi|^2\,d^3x-\int_{\R^3}\frac{\nu^2}{|x|}\,|\phi|^2\,d^3x
\ee
is a concave, bounded function of $\nu\in(0,1)$, for any $\phi\in C^\infty_0(\R^3)$, and so is its infimum with respect to $\phi$. A bounded concave function is continuous. 

Next, consider the function
\be{}
\phi(x):=\sqrt{\frac B{2\,\pi}}\,e^{-\frac{B}{4}(|x_1|^2+|x_2|^2)} \,\binom{f (x_3)}{0}\quad\forall\;x=(x_1,x_2,x_3)\in\R^3
\ee
for some $f\in C^\infty_0(\R,\R)$ such that $f\equiv 1$ for $|x|\leq\delta$, $\delta>0$, and \hbox{$\|f\|_{_{L^2(\R^+)}}=1$.} Note that $\phi\in {\rm Ker}(P_B+i\,\sigma_3\,\partial_{x_3})$ and so, $P_B\phi=-i\,\sigma_3\,\partial_{x_3}\phi$. Moreover, the function $\phi$ is normalized in $L^2(\R^3)$. Same computations as in \cite[proof of Proposition 6]{DEL} show that 
\be{}
\mathcal{E}_{B,\nu}[\phi]\leq\frac{C_1}{\nu}+C_2\,\nu- C_3\,\nu\log B\ ,
\ee
where $C_i$, $i=1$, $2$, $3$, are positive constants which depend only on $f$. For $B\geq 1$ large enough, $\mathcal{E}_{B,\nu}[\phi]<0$. With $\phi_{1/B}(x)=B^{-3/4}\,\phi\left(B^{-1/2}\,x\right)$, $\mathcal{E}_{1,\nu}[\phi_{1/B}]=B^{-1/2}\,\mathcal{E}_{B,\nu}[\phi]<0$ by~\eqn{Scaling}. This proves that $\mu(\nu)$ is negative. \end{proof}

\begin{lemma}\label{Lem7} For any $a$, $b\in\R^d$ and any $\varepsilon>0$,
\be{}
|a|^2\geq\frac{|a+b|^2}{1+\varepsilon}-\frac{|b|^2}\varepsilon\;.
\ee
\end{lemma}
\begin{proof} A simple computation shows that
\be{}
\varepsilon\,(1+ \varepsilon)\left[|a|^2-\frac{|a+b|^2}{1+\varepsilon}+\frac{|b|^2}\varepsilon\right]=(\varepsilon\,a-b)^2\geq 0\;.
\ee
\end{proof}

\subsection{Proof of Theorem \ref{Thm:bcritformula}}

Let us first prove that $B(\nu)\leq\tilde B(\nu):=4/\mu(\nu)^2$. By definition of $\mu(\nu)$ and \eqn{Scaling},
\be{formulabcrit}
\tilde B(\nu)=\sup\left\{B>0\;:\;\inf_\phi\(\mathcal{E}_{B,\nu}[\phi]+2\,\nrmphi\)\geq 0\right\}\;.
\ee
By definition of $B(\nu)$, $B<B(\nu)$ means that $\lambda_1(\nu,B)>-1$. Since $\lambda_1(\nu,B)$ is the infimum of $\lambda[\phi,\nu,B]$ with respect to $\phi\in C^\infty_0(\R^3)$ and $J[\phi,\lambda,\nu,B]$ is decreasing in $\lambda$,
\be{}
\mathcal{E}_{B,\nu}[\phi]+2\,\nrmphi\geq J[\phi,\lambda_1(\nu,B),\nu,B]\geq J[\phi,\lambda[\phi,\nu,B],\nu,B]=0
\ee
for all $\phi$, so that $B\leq \tilde B(\nu)$. This proves that $B(\nu)\leq \tilde B(\nu)$.

\medskip To prove the opposite inequality, let $B=B(\nu)$ and consider a sequence $(\nu_n)_{n\in\N}$ such that $\nu_n\in(0,\nu)$, $\lim_{n\to\infty}\nu_n=\nu$, $\lambda^n:=\lambda_1(\nu_n,B)>-1$ and $\lim_{n\to\infty}\lambda^n=-1$. Let~$\phi_n$ be the optimal function associated to~$\lambda^n$: $J[\phi_n,\lambda^n,\nu_n,B]=0$.

We define a sequence of truncation functions $(\chi_n)_{n\in\N}$ as follows. Consider first a nonnegative smooth radial function $\chi$ on $\R^+$ such that $\chi\equiv 1$ on $[0,1]$, $0\leq\chi\leq 1$ and $\chi\equiv 0$ on $[2,\infty)$. Then we set $\chi_n(x):=\chi(|x|/R_n)$ for some increasing sequence $(R_n)_{n\in\N}$ such that $\lim_{n\to\infty}R_n=\infty$. By applying Lemma~\ref{Lem7} to
\be{}
P_B\,\phi_n=\underbrace{(P_B\, \tilde\phi_n)\,\chi_n}_{=a}+\underbrace{\big[-(P_0\,\chi_n)\,\phi_n\big]}_{=b}\;,
\ee
where $\tilde\phi_n:= \phi_n\,\chi_n$, we get
\be{estimKin}
|P_B\,\phi_n|^2\geq \frac{|(P_B\, \tilde\phi_n)\,\chi_n|^2}{1+\varepsilon_n}-\frac{|(P_0\,\chi_n)\,\phi_n|^2}{\varepsilon_n}\;,
\ee
for some sequence $(\varepsilon_n)_{n\in\N}$ of positive numbers, to be fixed. Hence, using the fact that $0\leq\chi_n^2\leq 1$, we get
\begin{eqnarray}
&&\int_{\R^3}\frac{|P_B\,\phi_n|^2}{1+\lambda^n+\frac{\nu_n}{|x|}}\;d^3x\nonumber\\
&&\geq\frac 1{1+\varepsilon_n}\int_{\R^3}\frac{|P_B\, \tilde\phi_n|^2}{1+\lambda^n+\frac{\nu_n}{|x|}}\;d^3x-\frac 1{\varepsilon_n}\int_{\R^3}\frac{|(P_0\,\chi_n)\,\phi_n|^2}{1+\lambda^n+\frac{\nu_n}{|x|}}\;d^3x\;.
\end{eqnarray}
The function $\tilde\phi_n$ is supported in the ball $B(0,2\,R_n)$: with $\mu_n:=(1+\varepsilon_n)\big[2(1+\lambda^n)R_n+\nu_n\big]$,
\be{}
\frac 1{1+\varepsilon_n}\int_{\R^3}\frac{|P_B\, \tilde\phi_n|^2}{1+\lambda^n+\frac{\nu_n}{|x|}}\;d^3x\geq \frac 1{\mu_n}\int_{\R^3}|x|\,|P_B\, \tilde\phi_n|^2\,d^3x\;.
\ee
We choose $\varepsilon_n$ and $R_n$ such that
\be{}
\lim_{n\to\infty}\varepsilon_n=0\;,\quad\lim_{n\to\infty}R_n=\infty\quad\mbox{and}\quad\lim_{n\to\infty}(1+\lambda^n)\,R_n=0\;,
\ee
so that
\be{}
\lim_{n\to\infty}\mu_n=\nu\;.
\ee
The function $(P_0\,\chi_n)$ is supported in $B(0,2\,R_n)\setminus B(0,R_n)$: there exists constant $\kappa$ depending on $\|\chi'\|_{L^\infty(1,2)}$ such that $|P_0\,\chi_n|^2\leq\kappa\,R_n^{-2}$ and as a consequence,
\be{}
\frac 1{\varepsilon_n}\int_{\R^3}\frac{|(P_0\,\chi_n)\,\phi_n|^2}{1+\lambda^n+\frac{\nu_n}{|x|}}\;d^3x \leq \frac\kappa{\varepsilon_n\,R_n\big[(1+\lambda^n)\,R_n+\frac{\nu_n}2\big]}\int_{\R^3}|\phi_n|^2\;d^3x\;.
\ee
Moreover,
\be{}
\nu_n\int_{\R^3}\frac{|\phi_n|^2}{|x|}\;d^3x-\nu_n\int_{\R^3}\frac{|\tilde\phi_n|^2}{|x|}\;d^3x\leq\frac{\nu_n}{R_n}\int_{\R^3}|\phi_n|^2\,d^3x\;.
\ee
Thus, with $\eta_n:=\kappa/\!\(\varepsilon_n\,R_n\((1+\lambda^n)\,R_n+\frac{\nu_n}2\)\)+\nu_n/R_n$, we can write
\begin{eqnarray}
0=J[\phi_n,\lambda^n,\nu_n,B]&\geq& \frac 1{\mu_n}\int_{\R^3}|x|\,|P_B\, \tilde\phi_n|^2\,d^3x-\nu_n\int_{\R^3}\frac{|\tilde\phi_n|^2}{|x|}\;d^3x\nonumber\\
&&\qquad+(1-\lambda^n-\eta_n)\int_{\R^3}|\phi_n|^2\,d^3x\;.
\end{eqnarray}
Assume further that
\be{}
\lim_{n\to\infty}\varepsilon_n\,R_n=\infty\;,
\ee
so that $1-\lambda^n\geq\eta_n\to 0$ as $n\to\infty$. Using again the fact that $0\leq\chi_n^2\leq 1$, we get
\begin{eqnarray}
0=J[\phi_n,\lambda^n,\nu_n,B]&\geq&\frac 1{\mu_n}\int_{\R^3}|x|\,|P_B\, \tilde\phi_n|^2\,d^3x-\nu_n\int_{\R^3}\frac{|\tilde\phi_n|^2}{|x|}\;d^3x\nonumber\\
&&\qquad+(1-\lambda^n-\eta_n)\int_{\R^3}|\tilde\phi_n|^2\,d^3x\;.
\end{eqnarray}
Let $\tilde\nu_n=\sqrt{\mu_n\,\nu_n}$. We have obtained
\begin{eqnarray}
&&\frac 1{\tilde\nu_n}\int_{\R^3}|x|\,|P_B\, \tilde\phi_n|^2\,d^3x-\tilde\nu_n\int_{\R^3}\frac{|\tilde\phi_n|^2}{|x|}\;d^3x+2\int_{\R^3}|\tilde\phi_n|^2\,d^3x\nonumber\\
&&\qquad\leq\left[2-\sqrt{\frac{\mu_n}{\nu_n}}\,(1-\lambda^n-\eta_n)\right]\int_{\R^3}|\tilde\phi_n|^2\,d^3x\;.
\end{eqnarray}
The left hand side is a decreasing function of $\tilde\nu_n$. Since $\lim_{n\to\infty} \tilde\nu_n=\nu$, for any $\nu'>\nu$, for $n$ large enough,
\begin{eqnarray}
&&\frac 1{\nu'}\int_{\R^3}|x|\,|P_B\, \tilde\phi_n|^2\,d^3x-\nu'\int_{\R^3}\frac{|\tilde\phi_n|^2}{|x|}\;d^3x+2\int_{\R^3}|\tilde\phi_n|^2\,d^3x\nonumber\\
&&\qquad\leq \left[2-\sqrt{\frac{\mu_n}{\nu_n}}\,(1-\lambda^n-\eta_n)\right]\int_{\R^3}|\tilde\phi_n|^2\,d^3x\;.
\end{eqnarray}
We observe that by construction $\tilde\phi_n$ is non trivial. By homogeneity, one can even assume that $\|\tilde\phi_n\|_{L^2(\R^3)}=1$. Since $\lim_{n\to\infty}\sqrt{\frac{\mu_n}{\nu_n}}\,(1-\lambda^n-\eta_n)=2$,
\be{}
\tilde B(\nu')\leq B=B(\nu)\quad\forall\;\nu'>\nu\;.
\ee
By Lemma \ref{Lem6}, $\nu'\mapsto\tilde B(\nu')$ is continuous. This proves that $\tilde B(\nu)\leq B(\nu)$. \finprf

\section{A Landau level ansatz}\label{Sec:LandauLevels}

In analogy with what is done in nonrelativistic quantum mechanics, we denote by {\sl first Landau level\/} for a constant magnetic field of strength $B$, see \cite{DEL}, the space of all functions $\phi$ which are linear combinations of the functions
\be{}
\phi_\ell:=\frac B{\sqrt{2\,\pi\,{2^\ell\,\ell!}}}\,(x_2+i\,x_1)^\ell\,e^{-B\,{s^2}/{4}}\,{1\choose 0}\,,\quad\ell\in\N\,,\quad s^2=x_1^2+x_2^2\;,
\ee
where the coefficients depend only on $x_3$, {\sl i.e.,\/}
\be{}
\phi(x)=\sum_\ell f_\ell(x_3)\,\phi_\ell (x_1,x_2)\;.
\ee
In this section, we shall restrict the functional $\mathcal{E}_{B,\nu}$ to the first Landau level. In this framework, that we shall call the {\sl Landau level ansatz,\/} we also define a critical field by
\be{}
B_{\mathcal L}(\nu):=\inf\left\{B>0\;:\;\liminf_{b\nearrow B}\lambda_1^{\mathcal L}(\nu,b)=-1\right\}\,,
\ee
where
\be{}
\lambda_1^{\mathcal L}(\nu,B):=\inf_{\phi\in{\mathcal A}(\nu,B)\,,\;\Pi^\perp\phi=0}\lambda[\phi,\nu,B]\;.
\ee
Here $\Pi$ is the projection of $\phi$ onto the first Landau level, and $\Pi^\perp:=\mathbb I-\Pi$. 

One can prove in the Landau level ansatz a result which is the exact counterpart of Theorem~\ref{Thm:bcritformula}. For any $\nu\in(0,1)$, if
\be{}
\mu_{\mathcal L}(\nu):=\inf_{\phi\in{\mathcal A}(\nu,B)\,,\;\Pi^\perp\phi=0}\mathcal{E}_{1,\nu}[\phi]\;,
\ee
then
\be{Landau:bcritformula}
B_{\mathcal L}(\nu)=\frac 4{\mu_{\mathcal L}(\nu)^2}\;.
\ee
The goal of this section is to compare $\mu_{\mathcal L}(\nu)$ with $\mu(\nu)$ given by \eqn{Eqn:Equiv}. By definition of these quantities, we have
\be{frombelow}
\mu(\nu)\le\mu_{\mathcal L}(\nu)\;.
\ee

\medskip With the notation $s=\sqrt{x_1^2+x_2^2}$ and $z=x_3$, if $\phi$ is in the first Landau level, then 
\be{}
\mathcal{E}_{1,\nu}[\phi]=\sum_\ell\frac 1\nu\int_0^\infty b_\ell\,{f_\ell'}^2\;dz-\nu\int_0^\infty a_\ell\,f_\ell^2\;dz\;,
\ee
where
\be{}
a_\ell(z):=\left(\phi_\ell,\frac 1r\,\phi_\ell\right)_{L^2(\R^2,\C^2)}=\frac{1}{2^\ell\,\ell!}\,\int_0^{+\infty}\frac{s^{2\ell+1}\,e^{-s^2/2}}{\sqrt{s^2+z^2}}\;ds
\ee
and
\be{}
b_\ell(z):=\left(\phi_\ell, r\,\phi_\ell\right)_{L^2(\R^2,\C^2)}=\frac{1}{2^\ell\,\ell!}\,\int_0^{+\infty}{s^{2\ell+1}\,e^{-s^2/2}}{\sqrt{s^2+z^2}}\;ds\;.
\ee
A simple integration by parts shows that $\left(\phi_\ell, F(r)\,\phi_\ell\right)_{L^2(\R^2,\C^2)}$ is increasing (resp. decreasing) in $\ell$ whenever $F(r)$ is increasing (resp. decreasing). Since $a_\ell$ and $b_\ell$ only depend on $|z|$, $\mu_{\mathcal L}(\nu)$ is also achieved by functions which only depend on $|z|$ as well. It follows that
\begin{eqnarray}
&&\sum_\ell\(\frac 1\nu\int_0^\infty b_\ell\,{f_\ell'}^2\;dz-\nu\int_0^\infty a_\ell\,f_\ell^2\;dz\)\nonumber\\
&&\qquad\ge\;\frac 1\nu\int_0^\infty b_0\(\sum_\ell{f_\ell'}^2\)\,dz-\nu\int_0^\infty a_0\(\sum_\ell f_\ell^2\)\,dz\;.
\end{eqnarray}
Using the inequality
\be{}
\left|\;\frac{d}{dz}\,\textstyle{\sqrt{\sum_\ell f_\ell^2}}\;\right |^2\le\sum_\ell{\left|\;\frac{d}{dz}\,f_\ell\;\right|}^2\,,
\ee
we find that 
\be{}
\mathcal{E}_{1,\nu}[\phi]\geq\frac 1\nu\int_0^\infty b_0\,|f'|^2\,dz-\nu\int_0^\infty a_0\,f^2\,dz\;,
\ee
with $f:=\sqrt{\sum_\ell f_\ell^2}$. In other words, for our minimization purpose, it is sufficient to consider functions of the form
\be{}
\phi(x)=f(z)\,\frac{e^{-s^2/4}}{\sqrt{2\,\pi}}\,{1\choose 0}\,.
\ee
We observe that
\be{}
\frac 12\int_{\R^3}|\phi|^2\,d^3\vec x=\int_0^\infty f^2\,dz
\ee
and
\be{}
\frac12\,\mathcal{E}_{1,\nu}[\phi]=\frac 1\nu\int_0^\infty b\,{f'}^2\,dz-\nu\int_0^\infty a\,f^2\,dz:=\mathcal L_\nu[f]\;,
\ee
with $a=a_0$, $b=b_0$, {\sl i.e.,\/}
\be{}
b(z)=\int_0^\infty\sqrt{s^2+z^2}\,s\,e^{-s^2/2}\,ds\quad\mbox{and}\quad a(z)=\int_0^\infty\frac{s\,e^{-s^2/2}}{\sqrt{s^2+z^2}}\;ds\;.
\ee
The minimization problem in the Landau level ansatz is now reduced to
\be{}
\mu_{\mathcal L}(\nu)=\inf_f\frac{\mathcal L_\nu[f]}{\|f\|^2_{L^2(\R^+)}}\;.
\ee

It is a non trivial problem to estimate how close $\mu(\nu)$ and $\mu_{\mathcal L}(\nu)$ are. Let
\be{}
\mathcal L_\nu^-[f]:=\frac 1\nu\int_0^\infty\frac 1a\,{f'}^2\,dz-\nu\int_0^\infty a\,f^2\,dz
\ee
and
\be{}
\mathcal L_\nu^+[f]:=\frac 1\nu\int_0^\infty b\,{f'}^2\,dz-\nu\int_0^\infty\frac 1b\,f^2\,dz\;,
\ee
with corresponding infima $\mu_{\mathcal L}^-(\nu)$ and $\mu_{\mathcal L}^+(\nu)$.
\begin{lemma} For any $\nu\in(0,1)$,
\be{}
\mu_{\mathcal L}^-(\nu)\leq\mu_{\mathcal L}(\nu)\leq\mu_{\mathcal L}^+(\nu)\;.
\ee
\end{lemma}
\begin{proof} This follows from
\be{}
b(z)\ge\frac{1}{a(z)}
\ee
which in turn follows from Jensen's inequality, noting that $s\,e^{-s^2/2}\,ds/\,2\,\pi$ is a probability measure.
\end{proof}

In \cite{DEL} it was proved that
\be{asymp-}
\log|\mu_{\mathcal L}^-(\nu)|\approx -\frac{\pi}{2\,\nu}
\ee
as $\nu\to 0_+$. The methods in \cite{DEL} can be adapted to show that
\be{asymp+}
\log|\mu_{\mathcal L}^+(\nu)|\approx -\frac{\pi}{2\,\nu}
\ee
as well, thus proving the following result.
\begin{lemma} With the above notations, $\displaystyle\lim_{\nu\to 0_+}\nu\,\log|\mu_{\mathcal L}(\nu)|= -\frac{\pi}2$.\end{lemma}
As $\nu\to 0_+$, this provides us with the leading asymptotics of the critical magnetic fields.
\begin{theorem} With the above notations, $\displaystyle\lim_{\nu\to 0_+}\;\frac{\log B(\nu)}{\log B_{\mathcal L}(\nu)}=1$.\end{theorem}
\noindent{\bf Remark. }{\it This result does not prove that $\frac{B(\nu)}{B_{\mathcal L}(\nu)}$ converges to some finite limit as $\nu\to 0_+$.}\smallskip

\begin{proof} By \eqn{frombelow}, we already know that $\mu(\nu)\leq\mu_{\mathcal L}(\nu)$. To prove our result, we need an estimate of $\mu(\nu)$ from below. Since for $\nu$ small, $B(\nu)$ is very large according to \cite[Corollary 11]{DEL}, we can assume that $B(\nu)>1$ for any $\nu\in(0,\bar\nu)$ for some $\bar\nu>0$, so that, if $\nu\in(0,\bar\nu)$, then $\lambda_1(\nu,1)>-1$ and therefore, for all $\phi$,
\be{}
\mathcal{E}_{1,\nu}[\phi]\geq\mathcal{F}_\nu[\phi]:=\int_{\R^3}{\frac{|\sigma\cdot\nabla_1\phi|^2}{\lambda_1(\nu,1)+1+\frac\nu{|x|}}}\,\,d^3x-\int_{\R^3}\frac{\nu}{|x|}\,|\phi|^2\,d^3x\;.
\ee
For completeness, we sketch the main steps of the proof, which is similar to the one given in \cite{DEL}. Since
\be{}
\mathcal{G}_\nu{\phi \choose \chi}:=\left(H_B{\phi \choose \chi},{\phi \choose \chi}\right)
\ee
is concave in $\chi$,
\be{}
1+\mathcal{F}_\nu[\phi]=\sup_\chi\;\frac{\mathcal{G}_\nu{\phi \choose \chi}}{\|\phi\|_{L^2(\R^3)}^2+\|\chi\|_{L^2(\R^3)}^2}\;.
\ee
Obviously, 
\be{}
\sup_\chi\;\frac{\mathcal{G}_\nu{\phi \choose \chi}}{\|\phi\|_{L^2(\R^3)}+\|\chi\|_{L^2(\R^3)}}\geq\sup_{\Pi^\perp\chi=0}\;\frac{\mathcal{G}_\nu{\phi \choose \chi}}{\|\phi\|_{L^2(\R^3)}+\|\chi\|_{L^2(\R^3)}}\;.
\ee
This, by \cite[Proposition 14]{DEL}, is bounded below by
\be{58}
\sup_\chi\;\frac{\mathcal{G}_{\nu+\nu^{3/2}}\,\binom{\,\Pi\phi\,}{\Pi\chi} +\mathcal{G}_{\nu+\sqrt\nu}\,\binom{\,\Pi^\perp\phi\,}{0}}{\|\Pi\phi\|_{L^2(\R^3)}+\|\Pi^\perp\phi\|_{L^2(\R^3)}+\|\Pi\chi\|_{L^2(\R^3)}}
\ee
Moreover, by \cite[Proposition 15]{DEL}, for $\nu$ small enough, this is bounded below by
\be{59}
\sup_\chi\;\frac{\mathcal{G}_{\nu+\nu^{3/2}}\,\binom{\,\Pi\phi\,}{\Pi\chi} + d(\nu)\,\|\Pi^\perp\phi\|^2_{L^2(\R^3)}}{\|\Pi\phi\|_{L^2(\R^3)}+\|\Pi^\perp\phi\|_{L^2(\R^3)}+\|\Pi\chi\|_{L^2(\R^3)}}\,,
\ee
where $d(\cdot)$ is a continuous function such that $d(0)=\sqrt{2}$. The inequality that leads to~\eqn{59} displays the fact that being perpendicular to the lowest Landau level raises the energy.

Again by concavity, for all $\phi$, there is a unique $\chi$ realizing
\be{}
\sup_\chi\;\frac{\mathcal{G}_{\nu+\nu^{3/2}}\,\binom{\,\Pi\phi\,}{\Pi\chi}}{\|\Pi\phi\|_{L^2(\R^3)}+\|\Pi\chi\|_{L^2(\R^3)}}=:\lambda^-_{\nu+\nu^{3/2}}[\phi]\;.
\ee
Since $\lambda^-_{\nu+\nu^{3/2}}[\phi]<1<\sqrt{2}$, we finally have
\be{}
1+\mathcal{F}_\nu[\phi]\geq\lambda^-_{\nu+\nu^{3/2}}[\phi]\;.
\ee
By Theorem 16 and Corollary 17 of \cite{DEL},
\be{}
\lambda^-_{\nu+\nu^{3/2}}[\phi]=1+\mu_{\mathcal L}^-(\nu+\nu^{3/2})
\ee
and therefore, for $\nu$ small enough, we have proved that
\be{comparison}
\mu_{\mathcal L}^-(\nu+\nu^{3/2})\leq\mu(\nu)\le\mu_{\mathcal L}(\nu)\;,
\ee
where the upper estimate is given by Inequality~\eqn{frombelow}. Since $\nu\to 0_+$, $\nu^{3/2}$ becomes insignificant compared to $\nu$. The conclusion then holds by Theorem~\ref{Thm:bcritformula}. \end{proof}

\section{Numerical results}\label{Sec:Numerics}

\subsection{Computations in the Landau level ansatz}\label{Sec:NumLL}

To compute $\mu_{\mathcal L}$, we minimize $\mathcal L_\nu[f]/\|f\|^2_{L^2(\R^+)}$ on the set of the solutions $f_\lambda$ of
\be{}
{f}''+\frac{z\,a(z)}{b(z)}\,f'+\frac\nu{b(z)}\,(\lambda+\nu\,a(z))\,f=0\;,\quad
f(0)=1\;,\; f'(0)=0\;.
\ee
We notice that $b'(z)=z\,a(z)$, and, for any $z>0$,
\be{}
a(z)=e^{\frac{z^2}{2}} \sqrt{\frac{\pi }{2}}\,\mbox{\rm erfc}\left(\frac{z}{\sqrt{2}}\right)\quad\mbox{and}\quad b(z)=e^{\frac{z^2}{2}} \sqrt{\frac{\pi }{2}}\,\mbox{\rm erfc}\left(\frac{z}{\sqrt{2}}\right)+z\;.
\ee
Numerically, we use a shooting method and minimize $g(\lambda,z_{\rm max}):=|f_\lambda(z_{\rm max})|^2+|f'_\lambda(z_{\rm max})|^2\, $ for some $z_{\rm max}$ large enough. As $z_{\rm max}\to\infty$, the first minimum $\mu_{\mathcal L}(\nu,z_{\rm max})$ of $\lambda\mapsto g(\lambda,z_{\rm max})$ converges to $0$ and thus determines $\lambda=\mu_{\mathcal L}(\nu)$. See Figure 1.

\begin{figure}[ht]\begin{center}\includegraphics{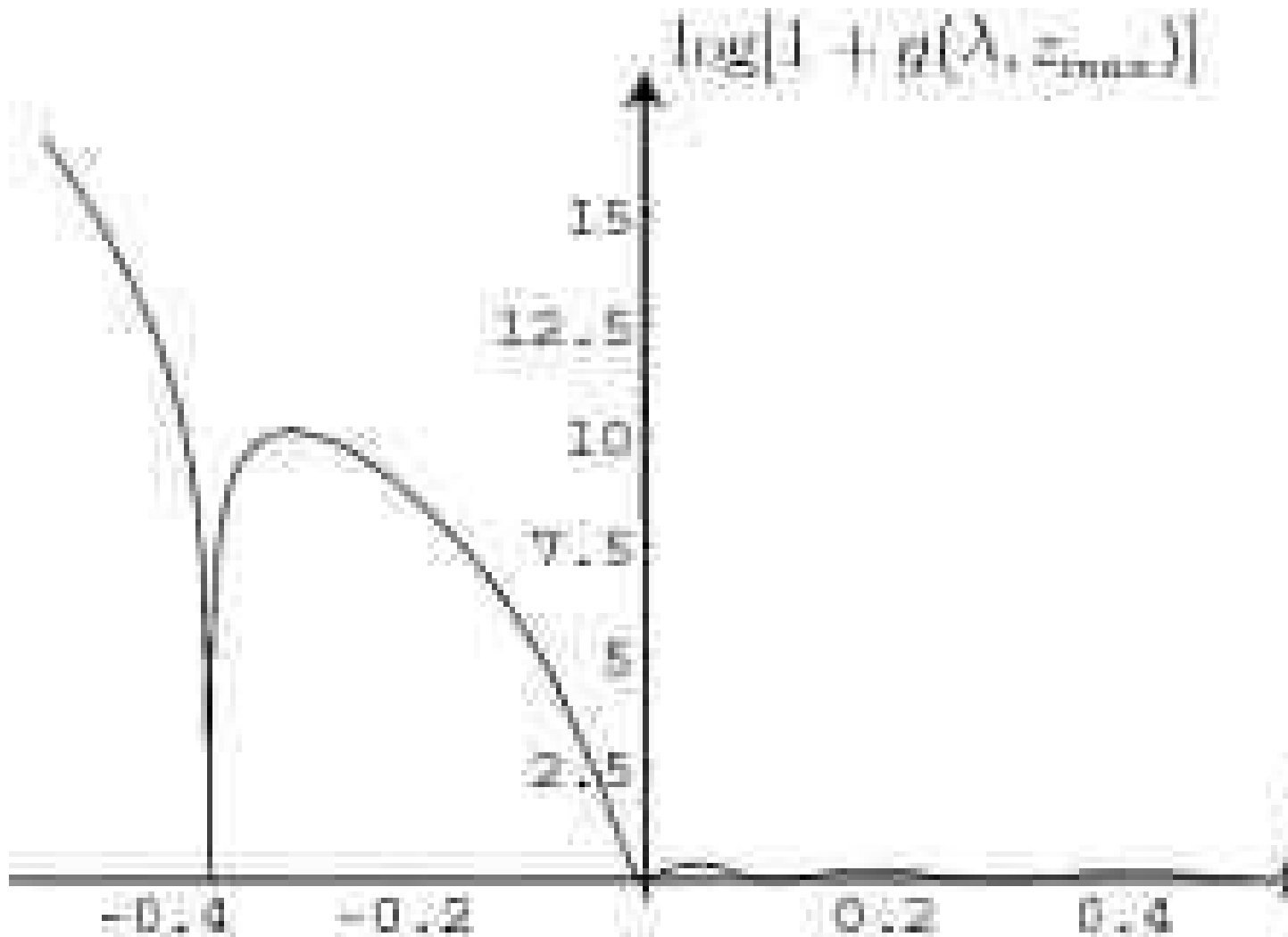}\caption{\sl Plot of $\lambda\mapsto\log[1+g(\lambda,z_{\rm max})]$ with $z_{\rm max}=100$, for $\nu=0.9$.}\end{center}\end{figure}

\medskip 
Let $b=\frac{m^2c^2}{e\,\hbar}\approx 4.414\cdot10^9$ be the numerical factor to obtain the critical field in Tesla. Corresponding values are given in $\log_{10}$ scale. The minimum $\mu_{\mathcal L}(\nu,z_{\rm max})$ is found by dichotomy. Results computed with Mathematica are given in Table 1. 

\begin{table}[htdp]
\footnotesize
\begin{center}
\begin{tabular}{|c|c|c|c|c|}
\hline
$\nu$&$Z$&$\mu_{\mathcal L}$&$B_{\mathcal L}(\nu)$&$\log_{10}(b\,B_{\mathcal L}(\nu))$\cr
\hline\hline
0.409&56.&-0.0461591&1877.35&12.9184\cr\hline
0.5&68.52&-0.0887408&507.941&12.3506\cr\hline
0.598&82.&-0.14525&189.596&11.9227\cr\hline
0.671&92.&-0.192837&107.567&11.6765\cr\hline
0.9&123.33&-0.363773&30.2274&11.1252\cr\hline
1&137.037&-0.445997&20.1093&10.9482\cr\hline\hline
\end{tabular}\vspace*{12pt}
\caption{\sl Numerical values found in the Landau level ansatz.}\end{center}
\end{table}

There is no significant difference with the results that were found in \cite{SchlueterEtAl}. See Fig. 2.

\begin{figure}[ht]\begin{center}\includegraphics[height=4.5cm]{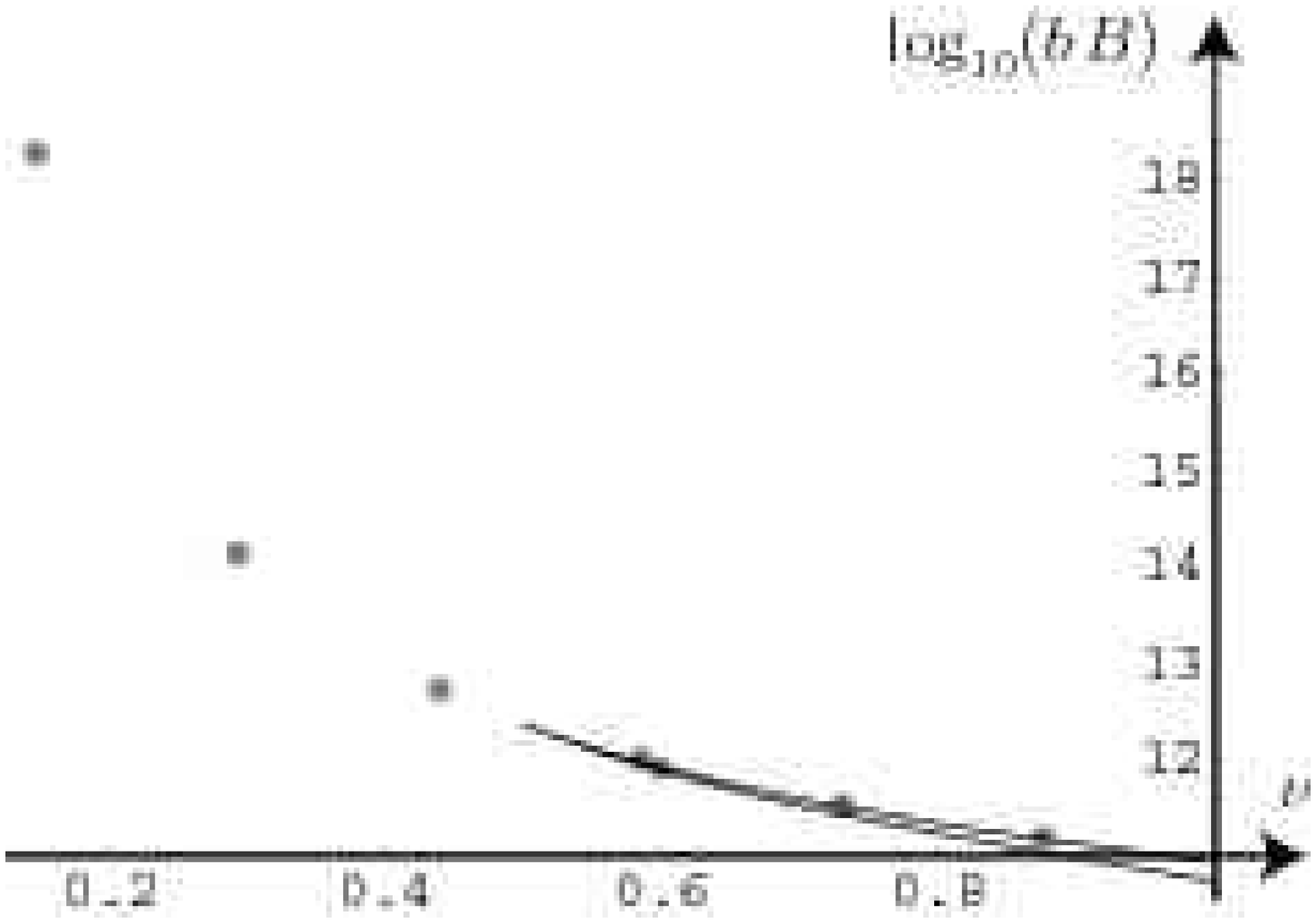}\hspace*{1cm}\includegraphics[height=4.5cm]{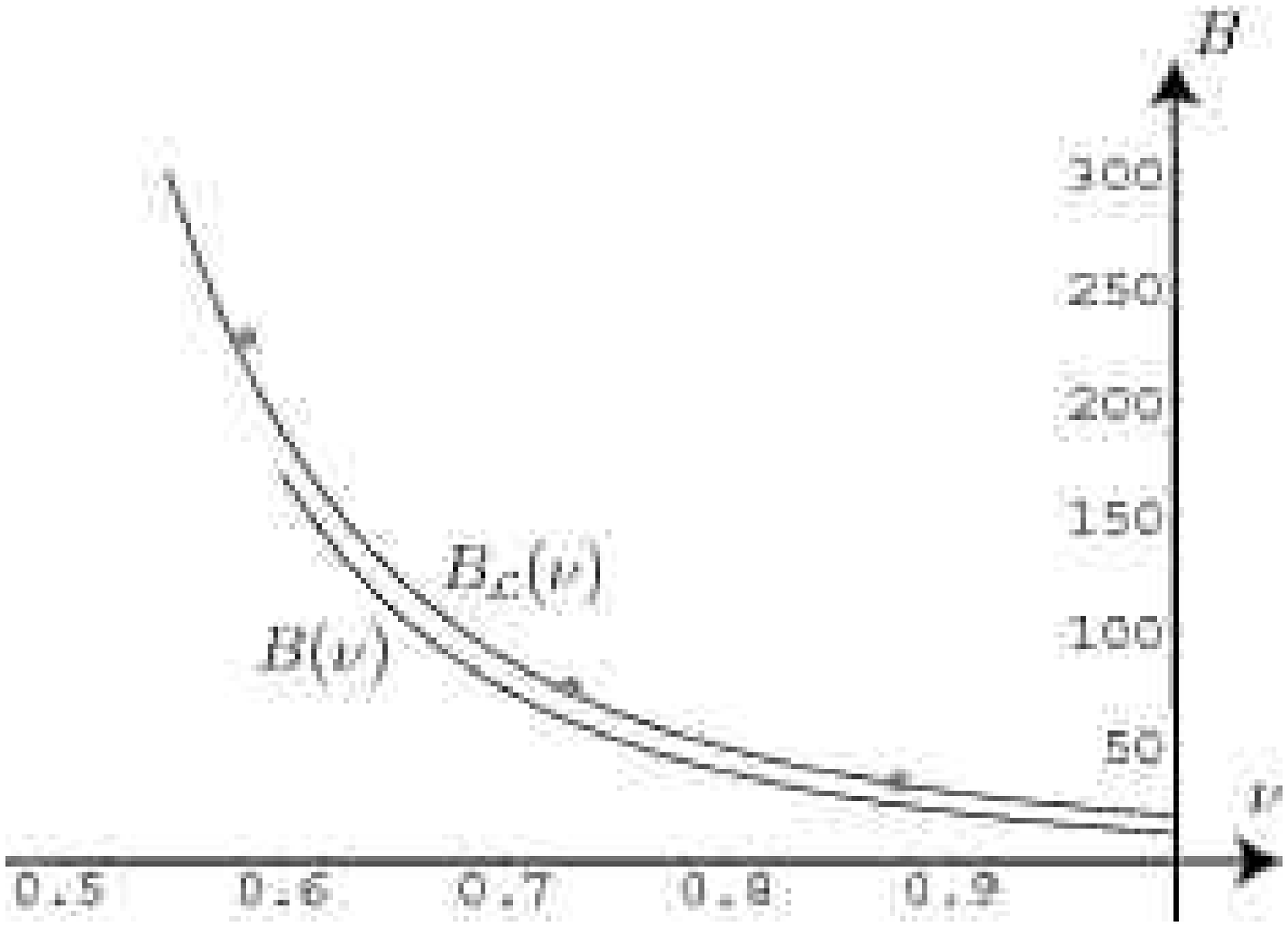}\caption{\sl Left: values of the critical magnetic field in Tesla ($\log_{10}$ scale). Right: values in dimensionless units. Ground state levels in the Landau level ansatz correspond to the upper curve, while the ones obtained without approximation are given by the lower curve. Dots correspond to the values computed by Schl\"uter, Wietschorke \& Greiner \cite{SchlueterEtAl} in the Landau level ansatz.}\end{center}\end{figure}

\subsection{Computations in the unconstrained case}\label{Sec:NumDirect}

We numerically compute $B(\nu)$ in the general case, without ansatz. For this purpose, we discretize the minimization problem defining $\mu(\nu)$ using B-spline functions of degree~$1$. They are defined on a logarithmic, variable step-size grid, in order to capture the behavior of the eigenfunctions near the singularity. We also use cylindrical symmetry to lower the dimension from $3$ to $2$. These two choices provide us with very sparse matrices, even if large. We use Matlab routines to calculate all integrals and the eigenvalues of the corresponding discretized matrices. The size of the computing domain is adapted as $\nu$ varies. Results of these computations are shown in Table 2.

\begin{table}[ht]
\footnotesize
\begin{center}
\begin{tabular}{|c|c|c|c|c|}
\hline
$\nu$&$Z$&$\lambda_1$&$B(\nu)$&$\log_{10}(b\,B(\nu))$\cr
\hline\hline
0.50&68.5185&-0.0874214&523.389&12.3637\cr\hline
0.55&75.3704&-0.119458&280.305&12.0925\cr\hline
0.60&82.2222&-0.153882&168.922&11.8725\cr\hline
0.65&89.0741&-0.191037&109.604&11.6847\cr\hline
0.70&95.9259&-0.231198&74.833&11.5189\cr\hline
0.75&102.778&-0.274665&53.0216&11.3693\cr\hline
0.80&109.63&-0.321875&38.6087&11.2315\cr\hline
0.85&116.481&-0.373535&28.668&11.1022\cr\hline
0.90&123.333&-0.430854&21.5476&10.9782\cr\hline
0.95&130.185&-0.496005&16.2588&10.8559\cr\hline
1.00&137.037&-0.573221&12.1735&10.7302\cr\hline\hline
\end{tabular}\vspace*{12pt}
\caption{\sl Results of the minimization method without symmetry ansatz.}\end{center}
\end{table}

\subsection{Discussion}\label{Sec:NumComments}

When dealing with the physics of magnetars, one is interested only in the order of magnitude of the critical magnetic field. With this goal in mind, the values computed in the Landau level ansatz are quite satisfactory. The corresponding values are given in the right column of Table 3, $\log_{10}(b\,B(\nu))$, where $b$ is approximatively $4.414\cdot10^9$ Tesla. See Fig.~2~(left).

Except maybe in the limit $\nu\to 0$, it is however clear from Theorem~\ref{Thm:bcritformula} that the equivalent eigenvalue problem~\eqn{Eqn:Equiv} has nothing to do with its counterpart $\mu_{\mathcal L}(\nu)$ in the Landau level ansatz. What our computations show is that the values of the computed critical fields significantly differ, see Fig.~3, and that the shapes of the corresponding ground state do not have much in common, see Fig. 4. 

\begin{figure}[ht]\begin{center}\includegraphics[height=5cm]{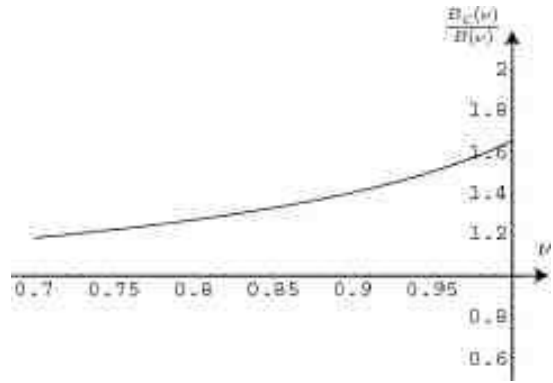}\caption{\sl Ratio of the ground state levels computed in the Landau level ansatz {\sl versus\/} ground state levels obtained by a computation without any symmetry ansatz.}\end{center}\end{figure}

\begin{figure}[ht]\begin{center}\includegraphics[height=4.5cm]{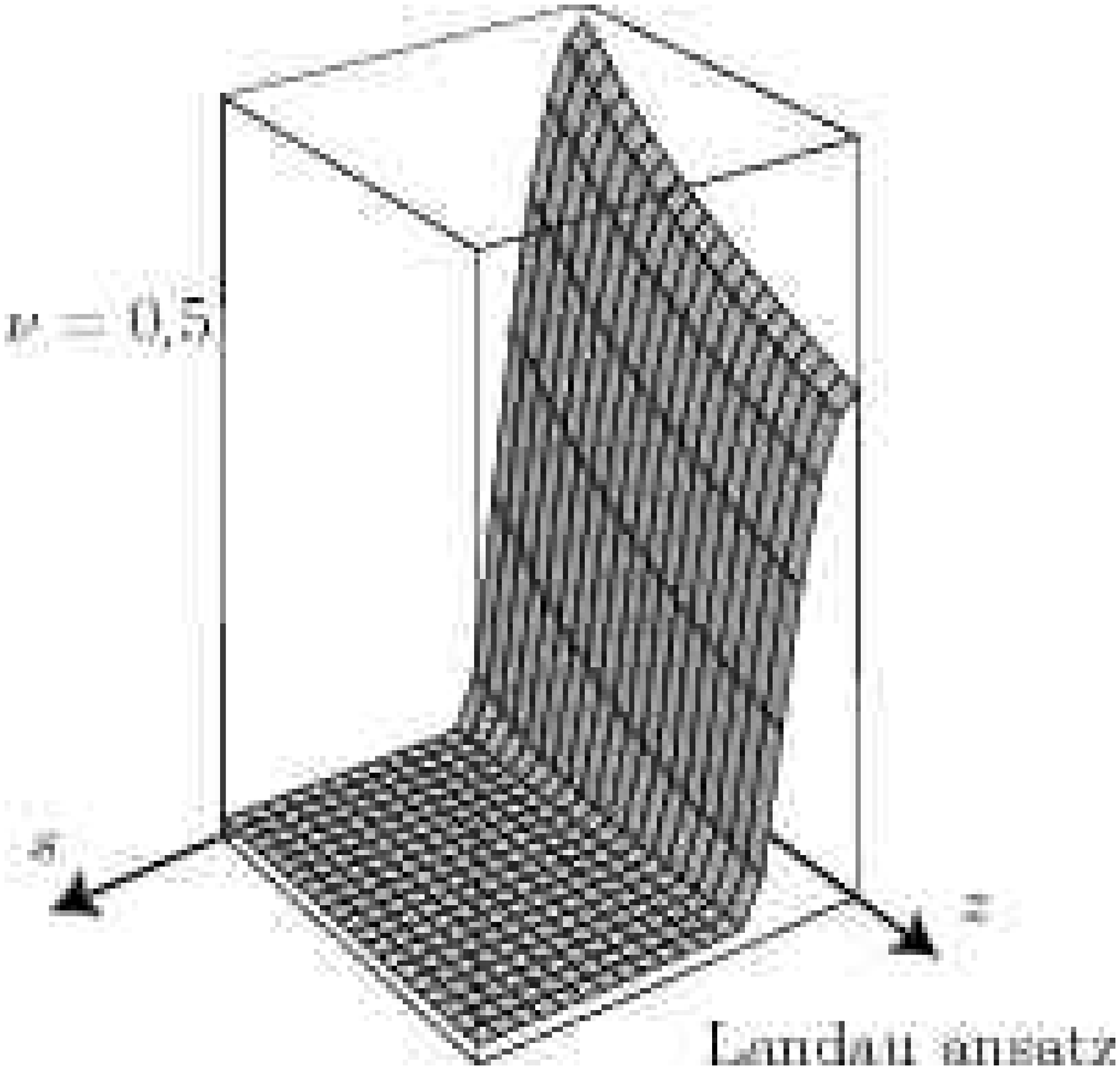}\hspace*{0.2cm}\includegraphics[height=4.5cm]{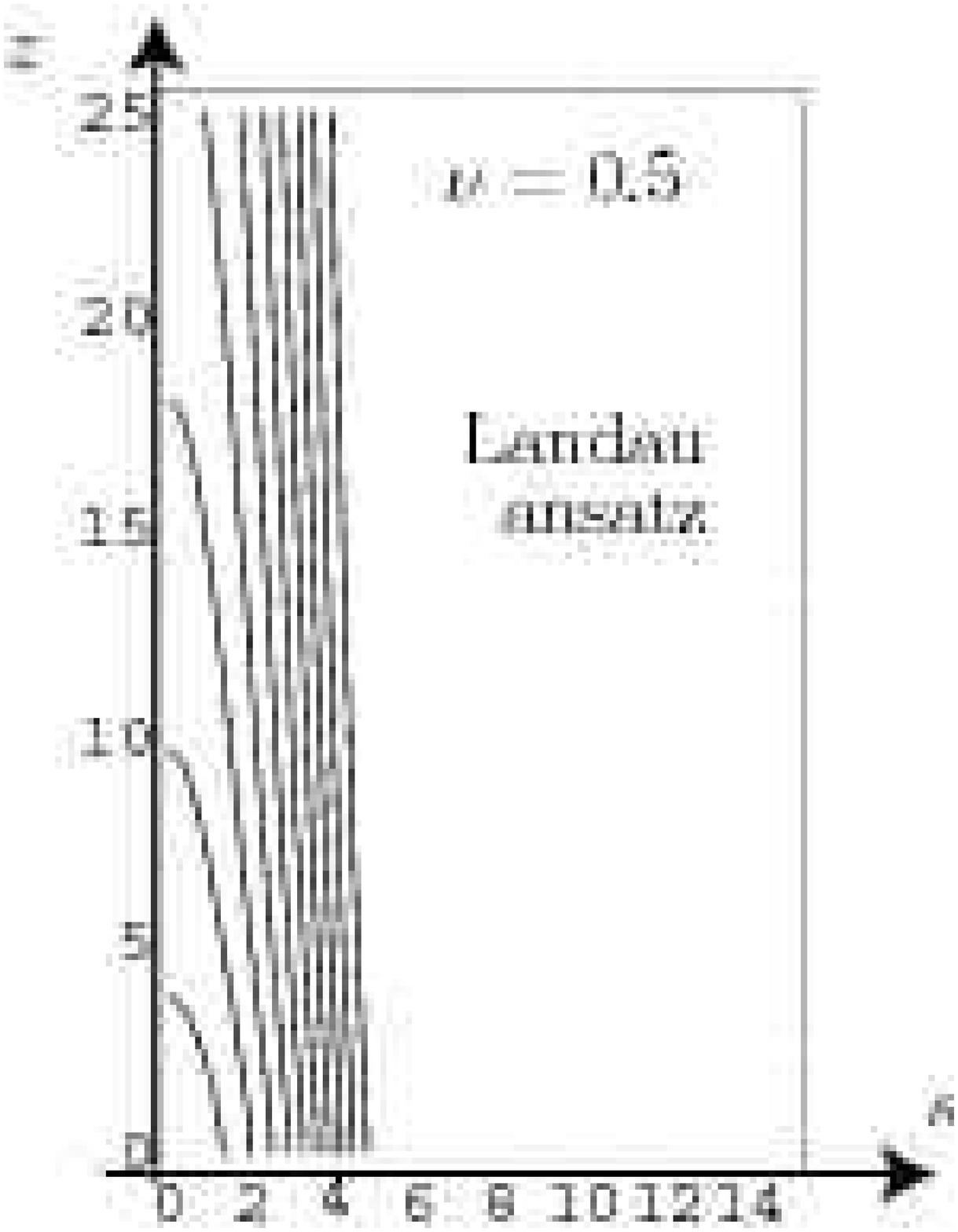}\hspace*{0.3cm}\includegraphics[height=4.5cm]{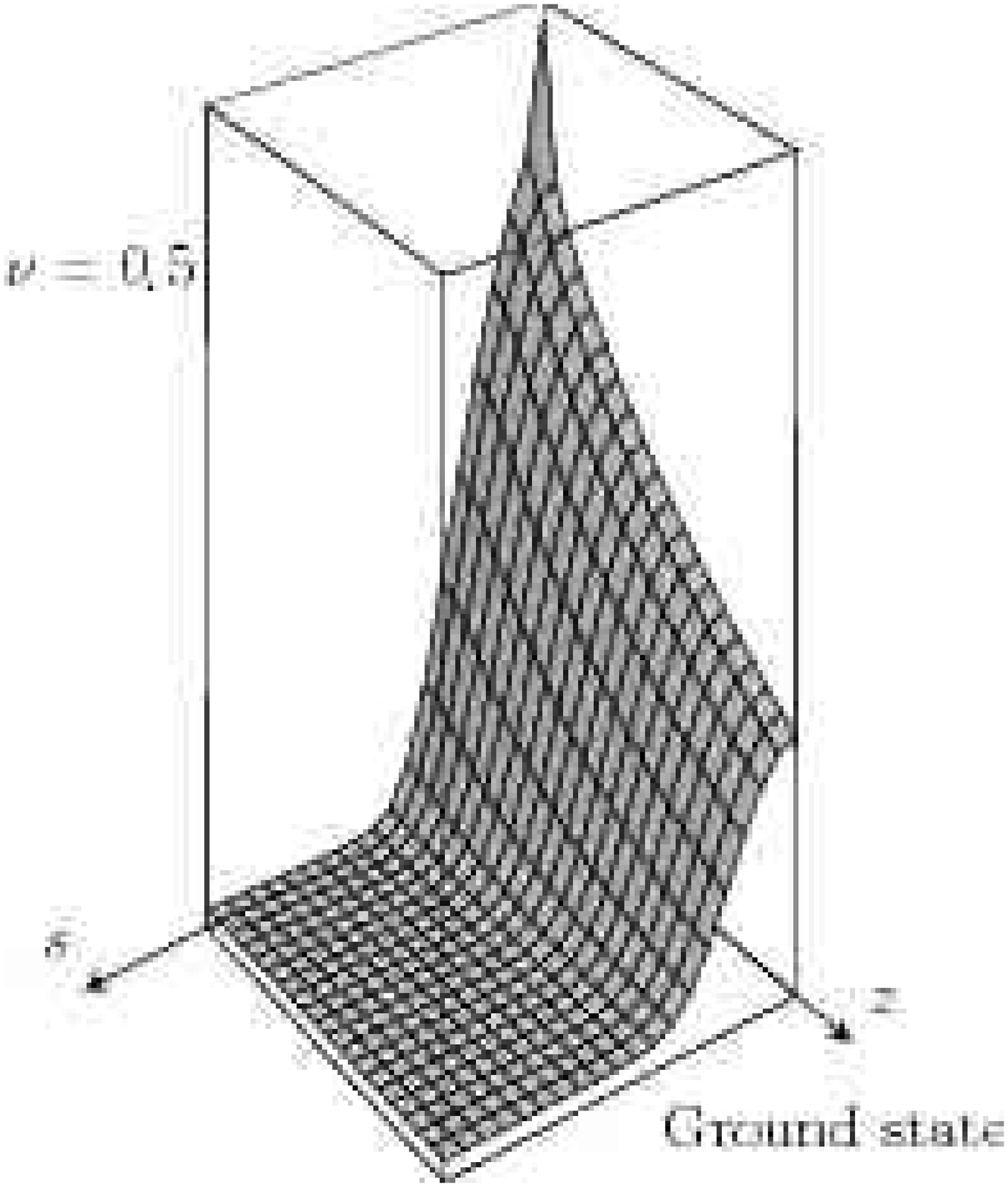}\hspace*{0.2cm}\includegraphics[height=4.5cm]{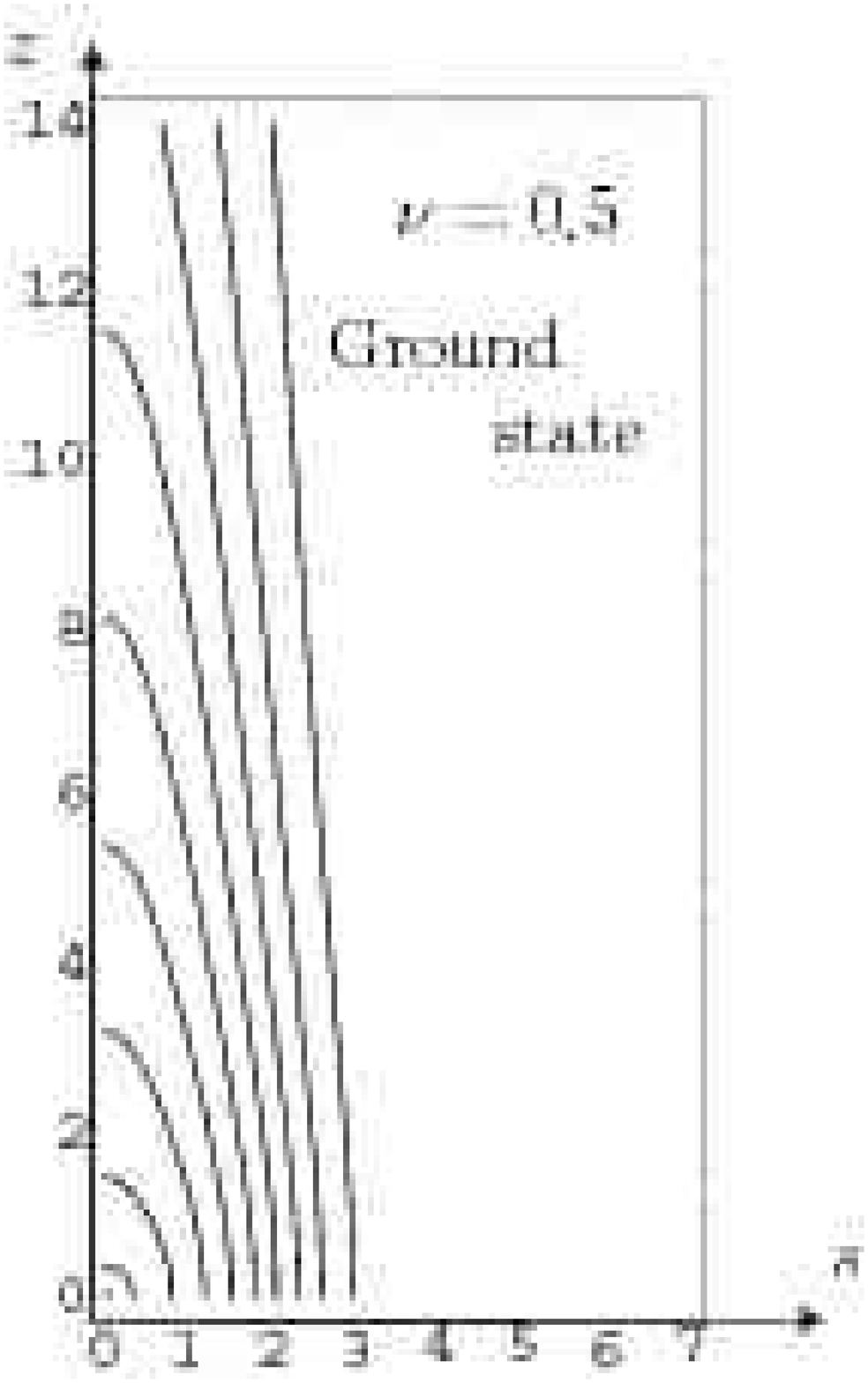}\newline\newline\hspace*{0.1cm}\includegraphics[height=4.5cm]{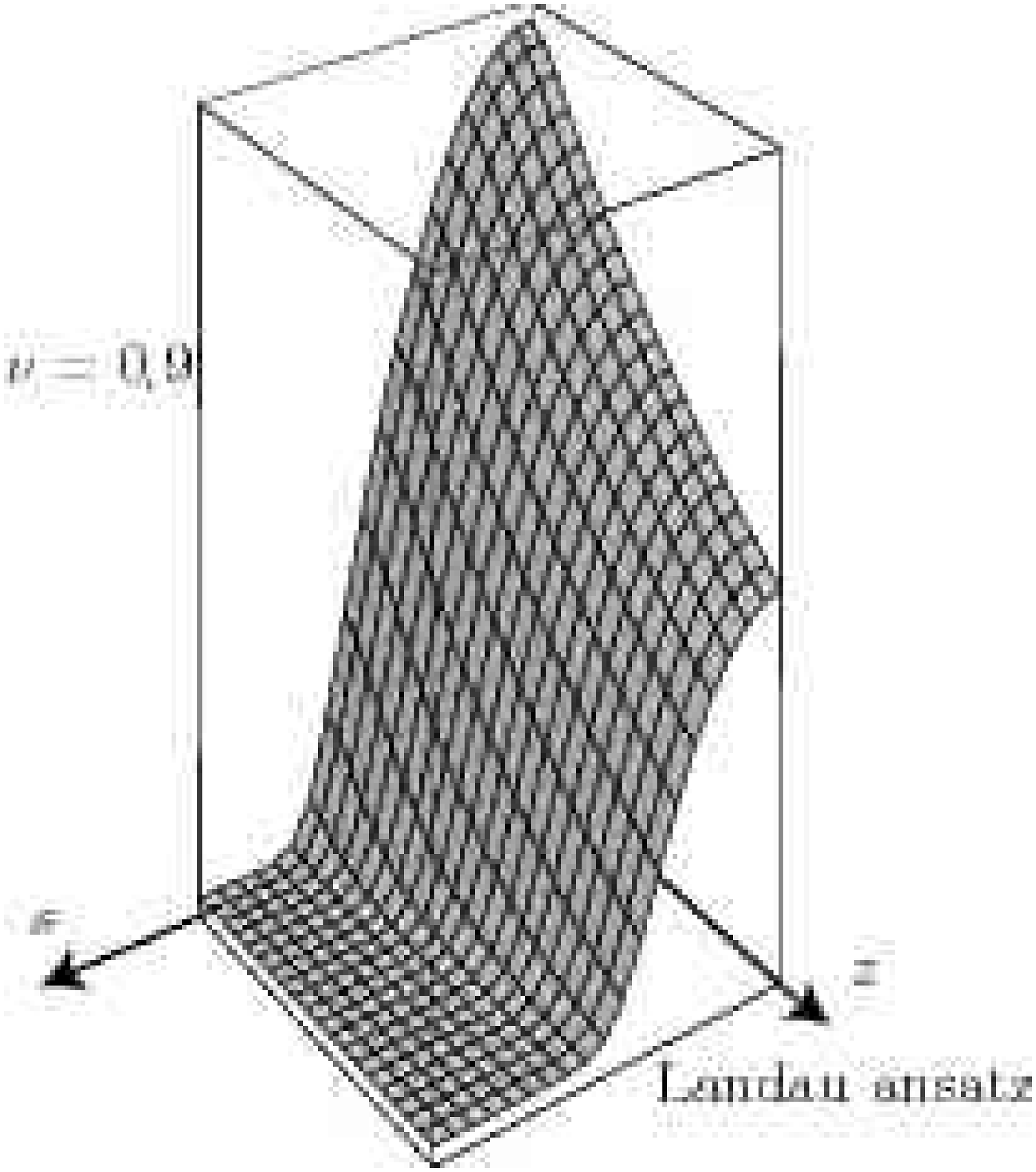}\hspace*{0.7cm}\includegraphics[height=4.5cm]{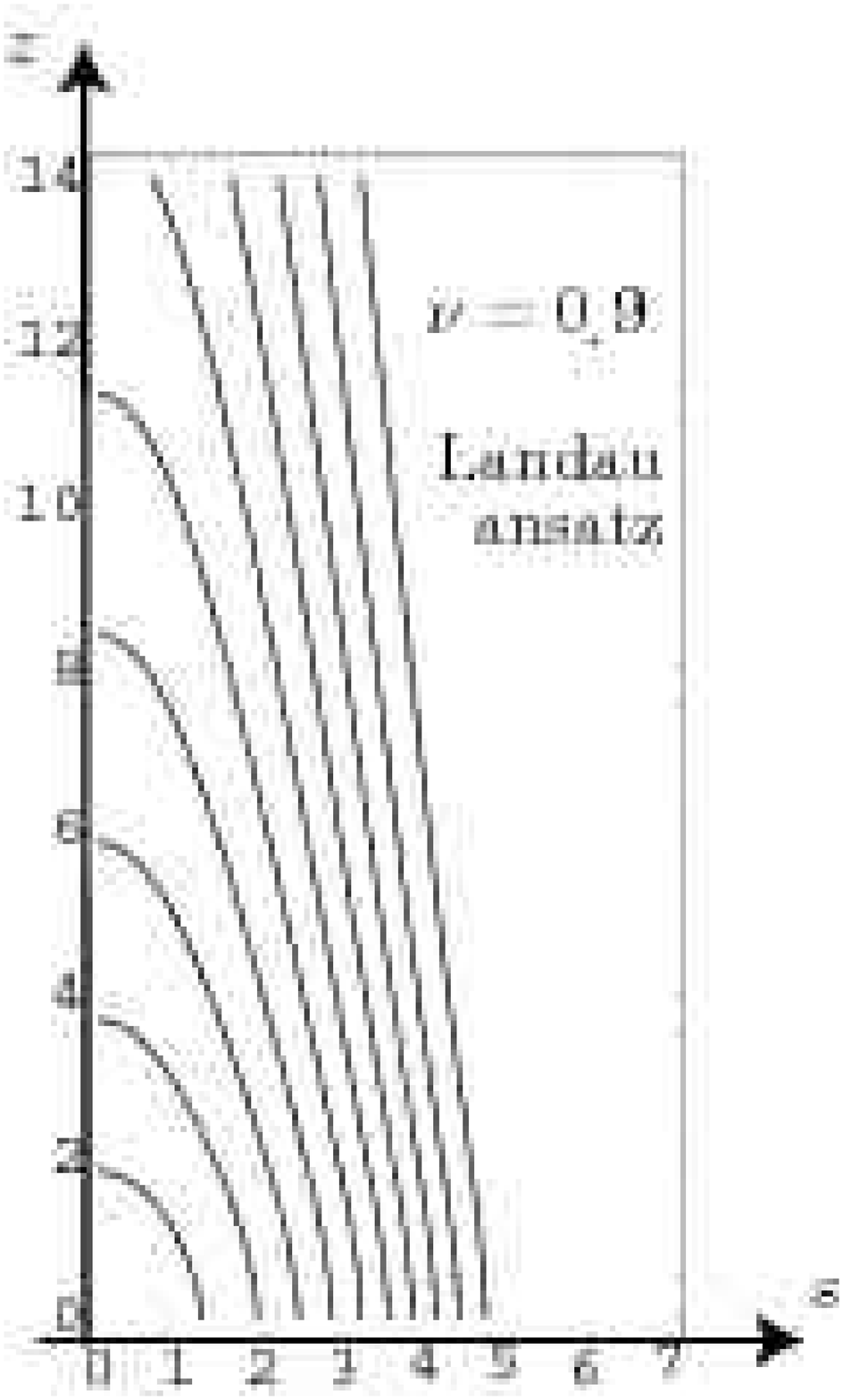}\hspace*{0.8cm}\includegraphics[height=4.5cm]{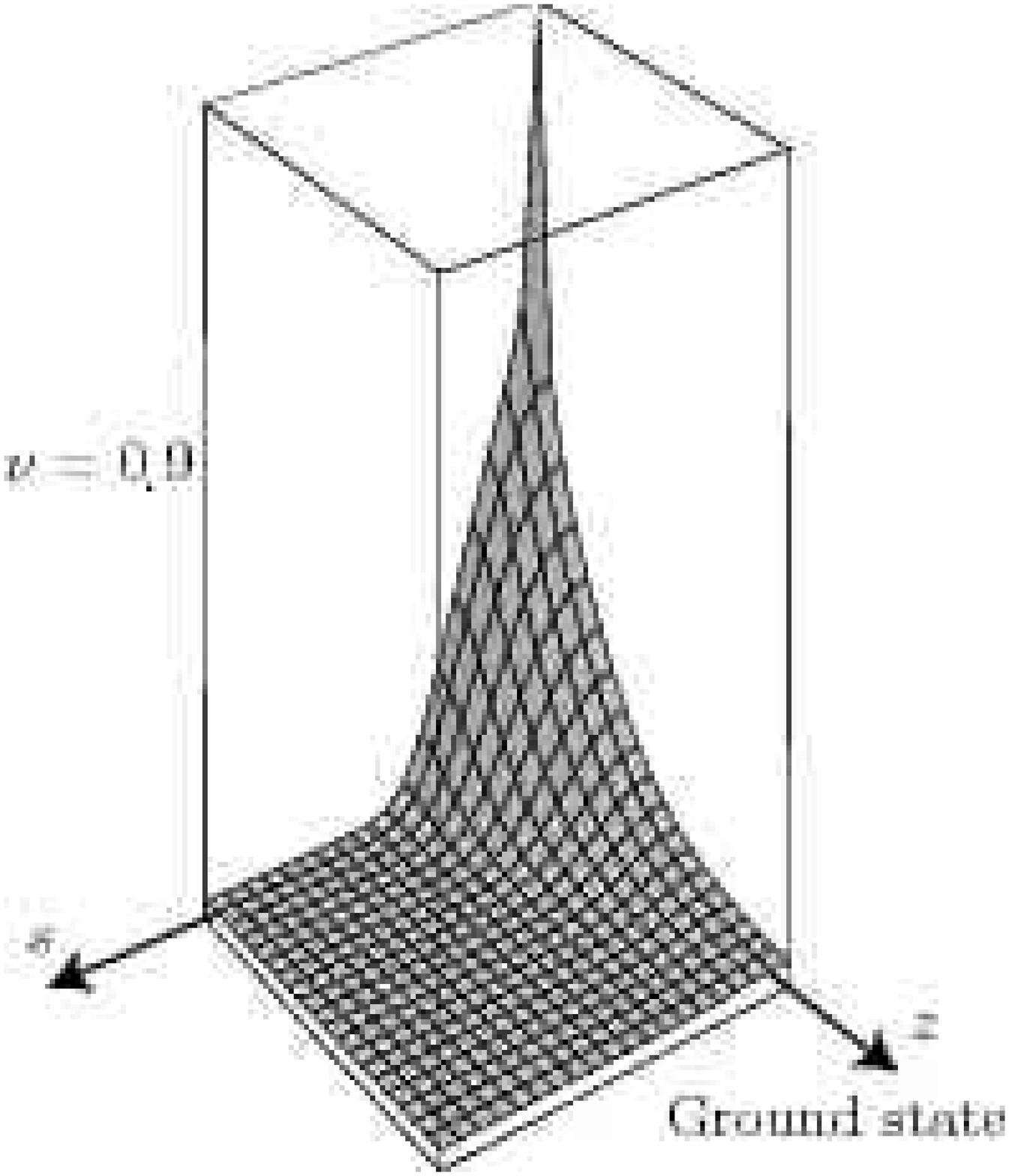}\hspace*{0.3cm}\includegraphics[height=4.5cm]{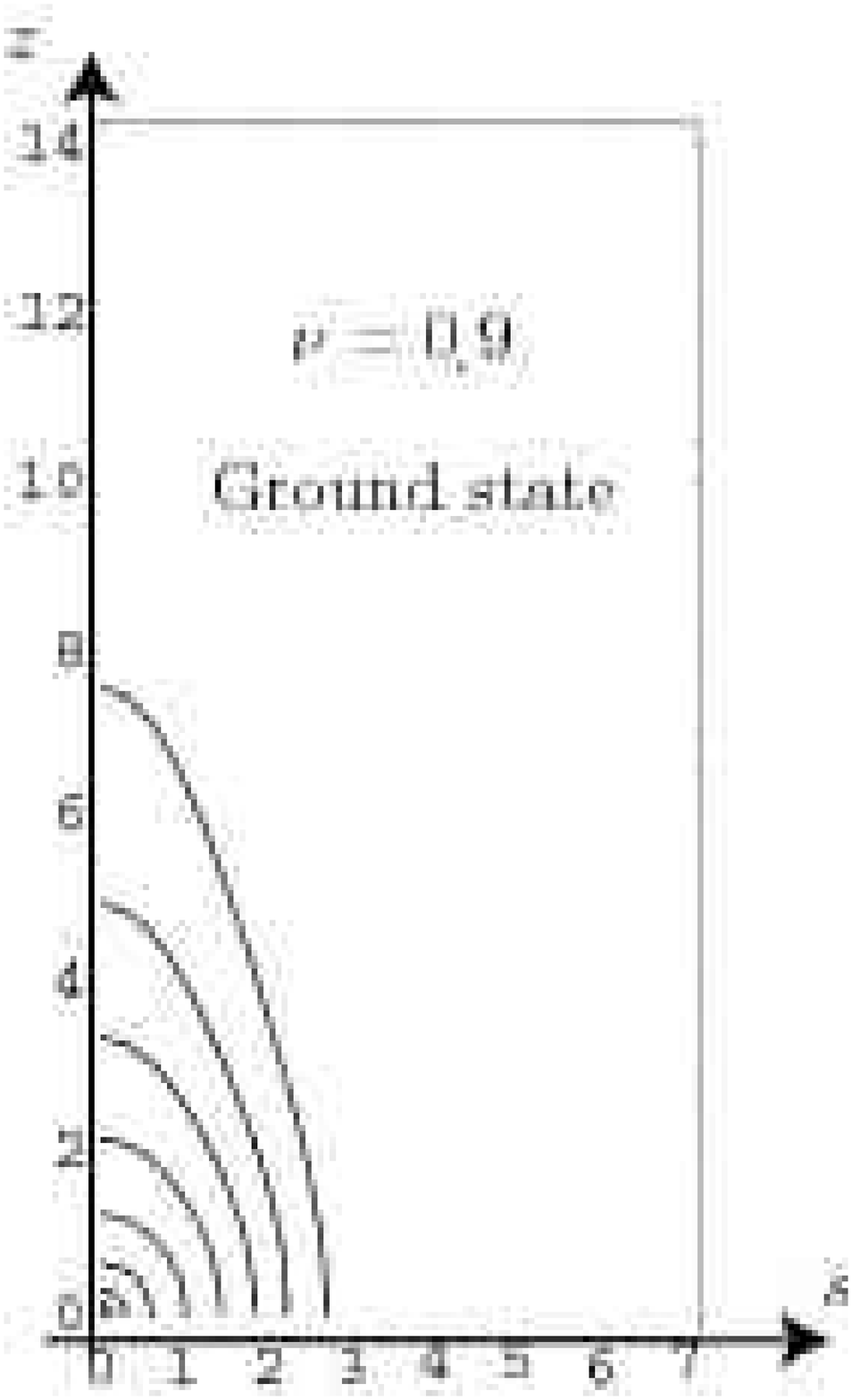}\caption {Plot of $\log_{\rm 10}(10^{-4}+|\phi|^4)$ for $\nu=0.5$ (top) and $\nu=0.9$ (bottom). Left: Landau level ansatz. Right: computation without any symmetry constraint. Level lines are equidistributed.}\end{center}\end{figure}

\section{Conclusion}\label{Sec:Conclusion}

Orders of magnitude of the critical magnetic field given by the Landau level ansatz are similar to the ones obtained without constraint. Qualitatively, the curve of the critical magnetic field in terms of the charge, is also well reproduced by the Landau level ansatz. 

However, the difference of the values of the critical magnetic fields is not small when comparison is done for values of $\nu$ approaching $1$. It turns out that values given by the Landau level ansatz are 50 \% higher than the ones found by a minimization approach without symmetry ansatz. Shapes of the density distributions also differ significantly. They are much more peaked around the singularity when computed unconstrained than in the Landau level ansatz.

The Landau level ansatz, which is commonly accepted in non relativistic quantum mechanics as a good approximation for large magnetic fields, is a quite crude approximation for the computation of the critical magnetic field (that is the strength of the field at which the lowest eigenvalue in the gap reaches its lower end) in the Dirac-Coulomb model. Even for small values of $\nu$, which were out of reach in our numerical study, it is not clear that the Landau level ansatz gives the correct approximation at first order in terms of $\nu$. Hence, accurate numerical computations involving the Dirac equation cannot simply rely on the Landau level ansatz.

\ack {J.D. and M.J.E. acknowledge support from ANR Accquarel project. M.L. is partially supported by U.S. National Science Foundation grant DMS 06-00037. The authors thank Joachim Reinhardt for pointing them \cite{SchlueterEtAl}.}

\medskip\noindent{\scriptsize\copyright~2007 by the authors. This paper may be reproduced, in its entirety, for non-commercial purposes.}

\end{document}